\documentclass[11pt,letterpaper]{article}
\usepackage[T1]{fontenc}
\usepackage[utf8]{inputenc}
\usepackage{lmodern}
\usepackage{microtype}
\usepackage{amsmath,amssymb,amsthm,mathtools}
\usepackage[margin=24mm,top=20mm,bottom=22mm,footskip=11mm]{geometry}
\usepackage{indentfirst}
\usepackage[hidelinks,pdfencoding=auto,psdextra]{hyperref}
\usepackage{bookmark}
\usepackage{setspace}
\allowdisplaybreaks[1]
\newtheorem{theorem}{Theorem}[section]
\newtheorem{lemma}[theorem]{Lemma}

\newtheorem{conjecture}[theorem]{Conjecture}
\newtheorem{claim}[theorem]{Claim}
\theoremstyle{definition}

\theoremstyle{remark}

\DeclareMathOperator{\bind}{bind}
\newcommand{\ov}[1]{\overline{#1}}

\makeatletter
\renewcommand{\@fnsymbol}[1]{\@arabic{#1}}
\makeatother

\hypersetup
{
	colorlinks=true,
	linkcolor=blue,
	filecolor=blue,
	urlcolor=blue,
	citecolor=blue,
}
\title{\vspace{-1.2em}\Large\bfseries
	A sharp bound on the binding number of \(K_k\)-free graphs and the Brandt conjecture}
\author{
	Jifu Lin\thanks{School of Mathematical Sciences, East China Normal University, Shanghai, China. E-mail: jifulin01@163.com}\quad\quad
	Heng Yang\thanks{Corresponding author. School of Mathematical Sciences, East China Normal University, Shanghai, China. E-mail: younghmath@163.com}
}
\date{}
\begin{document}
	\maketitle
	\vspace{-2.5em}

\begin{abstract}
The binding number of a graph is the smallest ratio $|N(S)|/|S|$ among all nonempty vertex sets $S$ whose neighborhood $N(S)$ is not the whole vertex set. We prove that every finite nonempty $K_k$-free graph has binding number at most $k-2$ for every $k\ge6$, confirming Brandt's conjecture; the bound is sharp for complete $(k-1)$-partite graphs with equal nonempty parts. We first prove the case $k=6$ and then extend it to all $k\ge6$ by induction.
\end{abstract}

\textbf{Keywords:} Binding number; $K_k$-free graph; Brandt's conjecture; Independence number

\textbf{MSC 2020:} 05C35; 05C76.

\section{Introduction}

All graphs considered are finite and simple. A graph $G$ is nonempty if $V(G)\ne\emptyset$.
For $S\subseteq V(G)$, let $N_G(S)=\bigcup_{v\in S}N_G(v)$.
This union of open neighborhoods may intersect $S$.
The binding number of $G$ is
\begin{equation}\label{eq:binding-definition}
	\bind(G)=
	\min\left\{
	\frac{|N_G(S)|}{|S|}
	:\ \varnothing\ne S\subseteq V(G),\ N_G(S)\ne V(G)
	\right\}.
\end{equation}
The binding number is well-defined because every singleton is admissible in a simple nonempty graph.

Building on Anderson's perfect-matchings result~\cite{Anderson}, Woodall~\cite{Woodall} introduced the binding number in 1973 as a neighborhood condition for the existence of matchings and cycles. 

Woodall \cite{Woodall} proved a Hamiltonicity theorem and asked whether the same binding condition forces a triangle. Shi \cite{ShiTriangle,ShiPancyclic} answered this question in 1985 and established the stronger pancyclicity conclusion in 1987.  

\begin{theorem}[{Woodall~\cite{Woodall}; Shi~\cite{ShiTriangle,ShiPancyclic}}]\label{thm:historical-cycles}
	Let $G$ be a graph with $\bind(G)\ge3/2$. Then\\
	{\rm (i)} $G$ is Hamiltonian.\\
	{\rm (ii)} $G$ contains a triangle.\\
	{\rm (iii)} $G$ is pancyclic, that is, it contains a cycle of every length from three to $|G|$.
\end{theorem}

These results show that a uniform neighborhood condition can force both a spanning structure and small prescribed subgraphs.

A complete $(k-1)$-partite graph with equal nonempty parts contains no $K_k$ and has binding number exactly $k-2$. Brandt conjectured that this construction gives the exact upper bound for all sufficiently large clique orders.

\begin{conjecture}[{Brandt~\cite{Brandt}; see Goddard~\cite{GoddardSurvey}}]\label{conj:brandt}
For all sufficiently large integers $k$, every nonempty $K_k$-free graph $G$
satisfies $
  \bind(G)\le k-2.$
\end{conjecture}

We next recall the known bounds and related results on Conjecture~\ref{conj:brandt}.

For triangles, a sufficient condition involving only independent sets is known. Robertshaw and Woodall~\cite{RW} proved that a graph $G$ of order $n$ contains a triangle if $|N_G(S)|>(n+|S|)/3$ for every nonempty independent set $S\subseteq V(G)$. Caro, in a postscript to the paper of Robertshaw and Woodall~\cite{RW}, proposed a conjecture that extends this condition: for an integer $k\ge 2$, every graph $G$ of order $n$ satisfying $|N_G(S)|>(k-2)(n+|S|)/k$ for every nonempty independent set $S$ contains a copy of $K_k$. Lyle and Goddard~\cite{LG} showed that the conjecture is false for every $k\ge 4$, by constructing, for each such $k$, a $K_k$-free graph that satisfies the stated hypothesis.

For larger cliques, Kane and Mohanty~\cite{KM} proved that $\bind(G)\ge k-4/3$ for an integer $k\ge 3$ implies that $G$ contains a copy of $K_k$. Lyle and Goddard~\cite{LG} subsequently improved this estimate: they constructed $K_4$-free graphs whose binding numbers are arbitrarily close to $(3+\sqrt{33})/4$, proved that every $K_4$-free graph satisfies $\bind(G)\le (5+\sqrt{91})/6$, and proved that for $k\ge 5$ every $K_k$-free graph satisfies $\bind(G)<k-3/2-{1}/({4k-6})$. In particular, the bound $k-2$ proposed in Conjecture~\ref{conj:brandt} does not hold for $k=4$, since these constructions yield $K_4$-free graphs with binding number greater than $2$; a restriction on the clique order in Conjecture~\ref{conj:brandt} is therefore necessary.

In this paper, we prove Conjecture~\ref{conj:brandt} by establishing the sharp upper bound for every $k\ge6$.  
\begin{theorem}\label{thm:main}
For every integer $k\ge6$ and every nonempty $K_k$-free graph $G$, we have $\bind(G)\le k-2$. For each $k\ge6$, the bound is sharp.
\end{theorem}

We use the following notation throughout.  The independence number $\alpha(G)$ and clique number $\omega(G)$ are the largest cardinalities of an independent set and a clique in $G$, respectively. We write $d_G(v)$ for the degree of a vertex $v$ in $G$. The minimum and maximum degrees are $\delta(G)$ and $\Delta(G)$. We define $N_G[v]=N_G(v)\cup\{v\}$ and $N_G[S]=\bigcup_{v\in S}N_G[v]$ for closed neighborhoods, $\ov G$ for the complement, and $|G|=|V(G)|$ for the order. Let $G$ and $H$ be graphs with disjoint vertex sets. Their disjoint union \(G\cup H\) has vertex set \(V(G)\cup V(H)\) and edge set \(E(G)\cup E(H)\), with no edges between \(V(G)\) and \(V(H)\). The join \(G\vee H\) is obtained from \(G\cup H\) by adding every edge between \(V(G)\) and \(V(H)\).

The rest of the paper is organized as follows. Section~\ref{sec:prelim} gives an equivalent characterization of $\bind(G)>c$, together with some lemmas from the literature and several lemmas of our own. Section~\ref{sec:private} introduces private cliques and the auxiliary lemmas. Section~\ref{mainresult} presents three auxiliary theorems, and we prove the main theorem at the end.

\section{Preliminaries}\label{sec:prelim}

A pair of nonempty vertex sets $(X,Y)$ in a graph $H$ is \emph{cross-complete} if $xy\in E(H)$ whenever $x\in X$, $y\in Y$, and $x\ne y$.  The sets may overlap; their intersection is then a clique. For $c\ge1$, define
$M_c(H)=\max\{c|X|+|Y|:(X,Y)\text{ is cross-complete in }H\}$ for nonempty $H$, and let $M_c(\varnothing)=0$.

\begin{lemma}\label{lem:binding-M}
For every graph $G$ of order $n$ and every $c\ge1$, $\bind(G)>c$ if and only if $M_c(\ov G)<n$.
\end{lemma}
\begin{proof}
Suppose $\bind(G)>c$.  For any cross-complete pair $(X,Y)$ in $\ov G$, $N_G(X)\subseteq V(G)\setminus Y$.  Since $Y\ne\varnothing$, we have $c|X|<|N_G(X)|\le n-|Y|$, and hence $M_c(\ov G)<n$. Conversely, for any nonempty $X$ with $N_G(X)\ne V(G)$, the pair $(X,V(G)\setminus N_G(X))$ is cross-complete in $\ov G$. Thus $M_c(\ov G)<n$ implies $c|X|<|N_G(X)|$, which proves $\bind(G)>c$.
\end{proof}

Passing to a subgraph can only decrease $M_c$. For $d\ge c\ge1$, we also have
\begin{equation}\label{eq:scaling}
  M_d(H)\ge\frac{d+1}{c+1}M_c(H).
\end{equation}
If $H$ is empty, \eqref{eq:scaling} is immediate. Otherwise, choose a maximizing pair for \(M_c(H)\) and label the sets so that \(|X|\ge|Y|\); this is possible because \(c\ge1\). Then $|X|\ge M_c(H)/(c+1)$, which gives \eqref{eq:scaling} by $M_d(H)\geq d|X|+|Y|=M_c(H)+(d-c)|X|$. Let $v$ be a vertex with $|N[v]|=D=\Delta(H)+1$, and let $Q$ be a maximum clique of size $a=\omega(H)$. The pairs $(N[v],\{v\})$ and $(Q,Q)$ are cross-complete and give
\begin{equation}\label{eq:elementary-M}
	M_c(H)\ge cD+1,\qquad M_c(H)\ge(c+1)a.
\end{equation}

The following neighborhood theorem is due to Robertshaw and Woodall~\cite[Theorem 4]{RW}: every nonempty triangle-free graph
$H$ has a nonempty independent set $Y$ such that $|N_H(Y)|\le(|H|+|Y|)/3$.

\begin{lemma}\label{lem:alpha2}
If $\alpha(H)\le2$, then $M_4(H)\ge5|H|/2$.
\end{lemma}
\begin{proof}
Apply the neighborhood theorem to the triangle-free graph $\ov H$, and let $X=V(H)\setminus N_{\ov H}(Y)$. Since $Y$ is an independent set of $\ov{H}$, we obtain that $Y\subseteq X$ and $(X,Y)$ is cross-complete in $H$. Moreover,  $3|X|+|Y|=3(|H|-|N_{\ov{H}}(Y)|)+|Y|\ge2|H|$.  Since $|X|\ge|Y|$, $4|X|+|Y|\ge\tfrac54(3|X|+|Y|)\ge5|H|/2$.
\end{proof}

The following triangle-free facts are contained in the triangle-free case of the Andr\'asfai--Erd\H{o}s--S\'os theorem and its proof \cite{AES}.

\begin{lemma}[Andr\'asfai--Erd\H{o}s--S\'os \cite{AES}]\label{lem:triangle-degree}
	Let $G$ be a triangle-free graph of order $m$.
	\\
	{\rm (i)} Every vertex of $G$ has at most two neighbors on a shortest
	odd cycle of $G$.\\
	{\rm (ii)} If $\delta(G)>2m/5$, then $G$ is bipartite.
\end{lemma}

A homomorphism \(G\to H\) is a map \(f:V(G)\to V(H)\) that preserves edges: whenever \(xy\in E(G)\), \(f(x)f(y)\in E(H)\). For $j\ge1$, define the Andr\'asfai graph $A_j$ on $\mathbb Z_{3j-1}$ by joining $i$ to $i+j,i+j+1,\ldots,i+2j-1$, with indices taken modulo $3j-1$. Thus $A_1=K_2$.  Each $A_j$ is $j$-regular.  Jin~\cite[p.~479]{Jin} instead uses neighbor differences $1,4,\ldots,3j-2$.  The relabelling $i\mapsto ji\pmod{3j-1}$ sends $1+3r$ to $j+r$ for $0\le r\le j-1$, since $3j\equiv1\pmod{3j-1}$. This explicitly identifies Jin's graph $F_j$ with $A_j$.

\begin{lemma}\label{lem:13over5}
If $G$ is triangle-free and nonbipartite, then $M_4(\ov G)\ge13|G|/5$.
\end{lemma}
\begin{proof}
Suppose $M_4(\ov G)<13m/5$, where $m=|G|$. Since $\Delta(\ov{G})+1=m-\delta(G)$, \eqref{eq:elementary-M} implies $\delta(G)>7m/20$. Jin \cite[Theorem 1]{Jin} gives a homomorphism from $G$ to the Andr\'asfai graph $A_6$.

Choose the smallest $j$ for which a homomorphism $f:G\to A_j$ exists. Then $2\le j\le6$, since a homomorphism to $K_2$ would make $G$ bipartite. By Durocher, Gunderson, Li and Skala~\cite[Lemma~2.1]{DGLS}, deleting any vertex of $A_j$ leaves a graph that has a homomorphism to $A_{j-1}$. Therefore every vertex of $A_j$ has a nonempty inverse image under $f$: otherwise, composing the map with this vertex-deletion homomorphism would contradict the choice of $j$.

For each $i\in V(A_j)$, let $Y_i=f^{-1}(i)$ and let $X_i=f^{-1}\bigl(V(A_j)\setminus N_{A_j}(i)\bigr)$. No edge of $G$ joins $X_i$ to $Y_i$, because such an edge would map to an edge incident with $i$ in $A_j$. Hence $(X_i,Y_i)$ is cross-complete in $\ov{G}$. Each vertex of $G$ belongs to exactly $2j-1$ of the sets $X_i$ and exactly one of the sets $Y_i$. Adding their values and dividing by $3j-1$ therefore yields
\[
M_4(\overline G)\ge\frac{\sum_i(4|X_i|+|Y_i|)}{3j-1}
=\frac{8j-3}{3j-1}m\ge\frac{13m}{5}.
\]
This contradicts our assumption.
\end{proof}

We use the following sufficient forms of Lyle's results. A $K_4$-free graph is maximal if adding any missing edge creates a $K_4$.

\begin{lemma}[Lyle~\cite{Lyle}]\label{lem:Lyle}
	Let $G$ be a maximal $K_4$-free graph of order $m$, and let $d=\delta(G)$.
	
	\noindent\hangafter=1\hangindent=1.5em
	{\rm (i)} \hspace{0em}If $d\ge4m/7$, then either $G=G_1\vee G_2$ such that $G_1$ has no edges and $G_2$ is triangle-free, or $\alpha(G)\ge4d-2m$.
	
	\noindent\hangafter=1\hangindent=2.4em
	{\rm (ii)} \hspace{0em}$\alpha(G)\ge\min\left\{\frac d2,\ \frac94d-m,\ \frac{5d-2m}{3},\ \frac52(2d-m)\right\}$.
	
\end{lemma}

\begin{proof}
	Part (i) is Lyle's Theorem 3. For (ii), choose a vertex of minimum degree in $G$. If the neighborhood contains no 5-cycle, Lyle's Lemma~9 gives $\alpha(G)\ge\min\{d/2,\,17d/7-m\}\geq \min\{d/2,\,9d/4-m\}$. If the neighborhood contains a 5-cycle, Lyle's Lemma~10 gives the minimum of the last three expressions in (ii).  
\end{proof}

\begin{lemma}\label{lem:alpha3}
	Let $H$ be a graph of order $m$ with $\alpha(H)\le3$, and let $D=\Delta(H)+1$.
	\\
	{\rm (i)} $M_4(H)\ge5m/3$.\\
	{\rm (ii)} If $m/D>5/2$, then $\omega(H)\ge m/3$.
\end{lemma}
\begin{proof}
Let $a=\omega(H)$. For (i), suppose $M_4(H)<5m/3$.
Inequalities~\eqref{eq:elementary-M} give $D<5m/12$ and
$a<m/3$.
Add edges to $\ov H$ to obtain a maximal $K_4$-free graph $G$.
Then $\delta(G)\ge m-D>7m/12>4m/7$ and
$\alpha(G)\le\alpha(\ov H)=a$.
The non-join alternative of Lemma~\ref{lem:Lyle}(i) gives
$\alpha(G)\ge4\delta(G)-2m>m/3>a$, a contradiction to
$\alpha(G)\le\omega(H)=a$.
In the join alternative, $G=G_1\vee G_2$ such that $G_1$ has
no edges and $G_2$ is triangle-free. Since $|G_1|\leq \alpha(G)\leq a<m/3$, we have $|G_2|>2m/3$. The graph $\overline{G_2}$ is a subgraph of $H$,
and $\alpha(\ov{G_2})\leq 2$. Lemma~\ref{lem:alpha2}
therefore implies $M_4(H)\geq M_4(\overline{G_2})\ge5|G_2|/2>5m/3$, again a contradiction.

For (ii), suppose $a<m/3$ and use the same maximal extension $G$.
Now $\delta(G)\ge m-D>3m/5>4m/7$.
The non-join alternative of Lemma~\ref{lem:Lyle}(i) gives $\alpha(G)>2m/5>a$, contradicting
$\alpha(G)\le\omega(H)=a$.
In the join alternative, we take $b=|G_1|\le a$.
Then $\delta(G_2)>3m/5-b>2(m-b)/5=2|G_2|/5$ because $b<m/3$.
Lemma~\ref{lem:triangle-degree}(ii) implies that $G_2$ is bipartite.
Hence $\alpha(G)\ge|G_2|/2>m/3>a$, again a contradiction.
\end{proof}

\section{Private cliques and auxiliary lemmas}\label{sec:private}

Let $H$ have order $n$, independence number $q$, clique number $a$, and $D=\Delta(H)+1$.  Fix a maximum independent set $I=\{v_1,\ldots,v_q\}$.  Define the corresponding \emph{private cliques} by $P_i=\{z\in V(H):N_H[z]\cap I=\{v_i\}\}$, and let $p_i=|P_i|$ and $e_i=D-p_i$. Each $P_i$ is a clique, since two nonadjacent vertices in $P_i$ could replace $v_i$ in $I$ and increase the cardinality of the independent set.

For $z\in V(H)$, let $r(z)=|N_H[z]\cap I|$.  Maximality of $I$ ensures $r(z)\ge1$.  Let $h=\sum_{r(z)\ge3}(r(z)-2)$.  By the definitions of $P_i$ and $h$, we have $\sum_z r(z)=2n-\sum_i p_i+h$. On the other hand, counting the same incidences at the vertices of $I$ implies $\sum_z r(z)=\sum_i|N_H[v_i]|\le qD$. Since $e_i=D-p_i$, $p_i\le a$ and $h\geq 0$, we have
\begin{equation}\label{eq:private-count}
	\sum_{i=1}^q e_i\le2qD-2n-h,\qquad
	e_i\ge D-a,\qquad a+D\ge\frac{2n}{q}.
\end{equation} If only $\alpha(H)\le q$ is known, the last inequality still holds, by applying the argument to a maximum independent set of cardinality at most $q$.

Suppose now that a maximum clique $Q$ is disjoint from $I$.  Call an index $i$ \emph{bad} if $P_i\cap Q\ne\emptyset$, and \emph{good} otherwise; we call $e_i$ a \emph{bad value} if $i$ is bad, and a \emph{good value} if $i$ is good. Let $B$ be the set of bad indices and let $q_i=|P_i\cap Q|$ for $i\in B$. If $z\in P_i\cap Q$, then $P_i\cup Q\subseteq N_H[z]$. Thus
$D\ge p_i+a-q_i$, which implies $q_i\ge a-e_i$. The private cliques are disjoint by their definition, so $\sum_{i\in B}q_i\le a$. Summing $q_i\ge a-e_i$ over $B$ and then using this upper bound yields
\begin{equation}\label{eq:bad-count}
	\sum_{i\in B}e_i\ge(|B|-1)a.
\end{equation}
Since $P_i\cup Q\subseteq N_H[z]$ for any $z\in P_i\cap Q$, we have $(P_i\cup Q,P_i\cap Q)$ is cross-complete in $H$. Hence, if $M_4(H)<\Lambda$, then $4(p_i+a-q_i)+q_i\leq M_4(H)<\Lambda$. Substituting $p_i=D-e_i$ yields \begin{equation}\label{eq:bad-weight}
	e_i>D-\frac{\Lambda}{4}+a-\frac34q_i\qquad(i\in B).
\end{equation}

\begin{lemma}\label{lem:two-exchange}
Suppose $z\in Q$ is adjacent to exactly $v_i,v_j$ in $I$, where $i,j$ are good.  Then $e_i+e_j\ge D$.
\end{lemma}
\begin{proof}
If $e_i+e_j<D$, then $p_i+p_j>D$.  Let $W_i=P_i\setminus N_H(z)$ and $W_j=P_j\setminus N_H(z)$. The vertex $z$ has at most $D-a$ neighbors outside $Q$. As $p_i,p_j\le a$, the inequality $p_i+p_j>D$ also gives $p_i,p_j>D-a$. Since $i,j$ are good, $P_i\cap Q=P_j\cap Q=\emptyset$.  Thus both $W_i,W_j$ are nonempty and $|W_i|+|W_j|\ge p_i+p_j-(D-a)>a$. If all cross-edges are present, $W_i\cup W_j$ is a clique larger than $Q$. Otherwise choose nonadjacent vertices $w_i\in W_i$ and $w_j\in W_j$. By the definition of the private cliques, $w_i,w_j$ are nonadjacent to $I\setminus\{v_i,v_j\}$. Since $z$ is adjacent to exactly $v_i,v_j$, $z$ has no neighbor in $I\setminus\{v_i,v_j\}$. Therefore, $(I\setminus\{v_i,v_j\})\cup\{w_i,w_j,z\}$ is an independent set of size $q+1$. Both possibilities are impossible.
\end{proof}

We say that $X$ dominates $Y$ if every vertex in $Y\setminus X$ has a neighbor in $X$.

\begin{lemma}\label{lem:three-exchange}
Let $J$ consist of three good indices, and let $E_J=\sum_{i\in J}e_i$.
Then $I\setminus\{v_i:i\in J\}$ dominates $Q$ under either of the following sets of assumptions:\\
{\rm (i)} $M_4(H)<\Lambda$, $a<\Lambda/5$, and $E_J<2D+a-2\Lambda/5$;\\
{\rm (ii)} $E_J<\min\{a+D/3,\,2D-a\}$.
\end{lemma}
\begin{proof}
Suppose $z\in Q$ has no neighbor in $I\setminus\{v_i:i\in J\}$. Since $I$ is a maximal independent set disjoint from $Q$, $z$ has at least one neighbor in \(\{v_i:i\in J\}\).  There cannot be just one such neighbor, since that neighbor would have a bad index.  Both (i) and (ii) imply $E_J<2D-a$.  If there are exactly two neighbors, the remaining index $\ell\in J$ satisfies $e_\ell\ge D-a$ by \eqref{eq:private-count}, so Lemma~\ref{lem:two-exchange} gives a contradiction.

Thus $z$ is adjacent to $v_i$ for every $i\in J$.  Let $W=\bigcup_{i\in J}(P_i\setminus N_H(z))$ and $w=|W|$. Then the three private cliques are disjoint and lie outside $Q$, and $z$ has at most $D-a$ neighbors outside $Q$. Since $e_i=D-p_i$, we have
$w\ge3D-E_J-(D-a)=2D+a-E_J$. Under the assumption (i), $w>2\Lambda/5$, and Lemma~\ref{lem:alpha2} rules out $\alpha(H[W])\le2$. Otherwise, $M_4(H)\geq M_4(H[W])\geq 5w/2>\Lambda$, a contradiction. Under the assumption (ii), $w>5D/3$ and $w>2a$.  If $\alpha(H[W])\le2$, $\ov{H[W]}$ is triangle-free with minimum degree at least $w-D>2w/5$, and Lemma~\ref{lem:triangle-degree}(ii) implies that $\ov{H[W]}$ is bipartite.  This gives a clique in $H[W]$ of cardinality at least $w/2>a$, also impossible.

Thus $H[W]$ has an independent set of three vertices, with one vertex in each of the three private cliques. By the definition of $W$, replacing \(\{v_i:i\in J\}\) by that independent set of three vertices and adding $z$ gives an independent set of $H$ of cardinality $q+1$, a contradiction.
\end{proof}

\begin{lemma}\label{lem:critical-edge}
Suppose $\alpha(H)=q$, the complement $\ov H$ is maximal $K_{q+1}$-free, and $Q$ is a maximum clique in $H$.  Suppose $zx\in E(H)$, where $z\in Q$ and $x\notin Q$. There is a maximum independent set $I$ disjoint from $Q$ for which at least one private clique meets $Q$.
\end{lemma}
\begin{proof}
Adding $zx$ to $\overline H$ creates a clique of size $q+1$. Equivalently, deleting $zx$ from $H$ creates an independent set
of cardinality $q+1$. The independent set contains both $z$ and $x$, since $H$ has independence number $q$. Remove the endpoint $z$ to obtain $I$.  Every remaining vertex other than $x$ is nonadjacent to $z$, and $I\cap Q=\emptyset$.  The vertex $z$ has just one neighbor in $I$, namely $x$, and lies in the private clique corresponding to $x$.
\end{proof}

\section{Proof of the main result}\label{mainresult}

Before proving the main result, we establish three auxiliary theorems.

\begin{theorem}\label{thm:alpha4}
If a graph $H$ satisfies $\alpha(H)\le4$, then $M_4(H)\ge5|H|/4$.
\end{theorem}
\begin{proof}
Suppose a counterexample exists.  Among all counterexamples of the same order, choose $H$ with as few edges as possible.
Deleting an edge cannot increase $M_4$, so deleting any edge of $H$ increases the independence number beyond four. Thus $\ov H$ is maximal $K_5$-free. Let $n=|H|$, $D=\Delta(H)+1$, and $a=\omega(H)$. The assumption $M_4(H)<5n/4=\Lambda$, together with \eqref{eq:elementary-M} and \eqref{eq:private-count}, gives
\begin{equation}\label{eq:a4-bounds}
 D<\frac{5n}{16},\qquad a<\frac n4,\qquad
 a\ge\frac n2-D>\frac{3n}{16}.
\end{equation}
A maximum independent set in $H$ dominates $H$; hence $n\le\alpha(H)D$, and $\alpha(H)=4$ since $D\leq 5n/16$.

If $r_1,r_2$ are nonadjacent vertices of $H$, the vertices outside $N_H[r_1]\cup N_H[r_2]$ induce a graph with independence number at most two since $\alpha(H)=4$.  Lemma~\ref{lem:alpha2} and $M_4(H)<5n/4$ show that fewer than $n/2$ vertices lie outside $N_H[r_1]\cup N_H[r_2]$. Consequently, $|N_H[r_1]\cap N_H[r_2]|=|N_{H}[r_1]|+|N_{H}[r_2]|-|N_H[r_1]\cup N_H[r_2]|<2D-n/2$. For a maximum clique $Q$ in $H$ and $x\notin Q$, maximality of $Q$ provides a vertex $y\in Q$ nonadjacent to $x$. Since $N_H(x)\cap Q\subseteq N_H[x]\cap N_H[y]$, writing $L=2D-n/2$ gives
\begin{equation}\label{eq:a4-external}
 |N_H(x)\cap Q|<L<a.
\end{equation}
Here $L<a$ follows from $a\ge n/2-D$ and $D<5n/16<n/3$.

\begin{claim}\label{cl:a4-forbidden}
If $\alpha(H[U])\leq 2$, then $|U|<n-2D+a/2$.
\end{claim}
\begin{proof}[Proof of claim~\ref{cl:a4-forbidden}]
Suppose such a set $U$ has $|U|\ge n-2D+a/2$. Choose a vertex set $S$ of largest cardinality such that $\alpha(H[S])\leq 2$, and let $s=|S|$. Then $\ov{H[S]}$ is triangle-free and $s\geq n-2D+a/2$. If $\ov{H[S]}$ is bipartite, then $a\ge s/2\ge(n-2D+a/2)/2$, which gives $a\ge2(n-2D)/3>n/4$, contrary to \eqref{eq:a4-bounds}. Thus $\ov{H[S]}$ is nonbipartite. Lemma~\ref{lem:13over5} and $M_4(H)<5n/4$ give $s<25n/52$. On the other hand, $s\ge n-2D+a/2>3n/8+a/2$ by $D<5n/16$; hence $a<11n/52$.

Choose a shortest odd cycle $C$ in $\ov{H[S]}$. By Lemma~\ref{lem:triangle-degree}(i), every vertex of $S$ is adjacent in $\ov H$ to at most two vertices of $C$.  Therefore, $\sum_{bc\in E(C)}|S\cap(N_{\ov H}(b)\cup N_{\ov H}(c))|\le4s$. As \(|C|\ge 5\), some edge $bc$ of $C$ satisfies that at least \(s/5\) vertices in $S$ have no neighbor among \(\{b,c\}\) in \(\overline H\). Then $N_{\overline H}(b)\cap N_{\overline H}(c)$ is triangle-free in $\overline H$ because $\overline H$ is $K_5$-free, and it is empty in $\overline{H[S]}$ since $\overline{H[S]}$ is triangle-free. Since $\delta(\ov H)=n-D$, we have $|N_{\overline H}(b)\cap N_{\overline H}(c)|\geq 2(n-D)-(n-s/5)=n-2D+s/5$. Choose $T\subseteq V(H)\setminus S$ of largest cardinality such that $\ov{H[T]}$ is triangle-free, and let $t=|T|$. Then $t\ge n-2D+s/5$.

If $S\cup T=V(H)$, Lemma~\ref{lem:alpha2} applied to $H[S]$ or $H[T]$ gives $M_4(H)\ge5n/4$, a contradiction. Otherwise choose $v\notin S\cup T$. Adding $v$ to $S$ or to $T$ creates a triangle in $\ov H$ by the choice of $S$ and $T$. Hence there are edges $b_1b_2$ in $N_{\ov H}(v)\cap S$ and $c_1c_2$ in $N_{\ov H}(v)\cap T$. For each $i\in\{1,2\}$, the intersection $N_{\overline H}(b_i)\cap S$ is independent in $\overline H$; as $\overline{H[S]}$ is triangle-free, $|N_{\overline H}(b_i)\cap S|\le \alpha(\overline H)=a$. The intersection $N_{\overline H}(v)\cap N_{\overline H}(c_1)\cap N_{\overline H}(c_2)\cap S$ is independent in $\overline H$, since $\overline H$ is $K_5$-free. At most $a$ vertices of $S$ are therefore adjacent to all three, and every other vertex is adjacent to at most two. Thus the sum of the numbers of neighbors in $S$ of $v,c_1,c_2$ is at most $2s+a$. The sum of the numbers of neighbors in $S$ of the five vertices $v,b_1,b_2,c_1,c_2$ is at most $2s+3a$. Similarly, the corresponding sum in $T$ is at most $2t+3a$. Thus $5(n-D)\leq d_{\ov H}(v)+d_{\ov H}(b_1)+d_{\ov H}(b_2)+d_{\ov H}(c_1)+d_{\ov H}(c_2)\leq (2s+3a)+(2t+3a)+5(n-s-t)=5n-3(s+t)+6a$. On the other hand, $s+t\geq n-2D+6s/5\geq 11(n-2D)/5+3a/5$ by $s\geq n-2D+a/2$ and $t\geq n-2D+s/5$.

Combining the previous two inequalities, we obtain $33n-91D\leq 21a$, and thus $a\ge(33n-91D)/21>73n/336>11n/52$,
contrary to $a<11n/52$.  This proves the claim.
\end{proof}

Let $Q$ be a maximum clique in $H$. If no edge of $H$ joins $Q$ to $V(H)\setminus Q$, then $\alpha(H-Q)\le3$ since $\alpha(H)=4$. Lemma~\ref{lem:alpha3}(i) gives $M_4(H)\ge5(n-a)/3>5n/4$ by $a<n/4$, a contradiction. Thus such an edge exists. Lemma~\ref{lem:critical-edge} therefore supplies a maximum independent set $I$ in $H$ disjoint from $Q$, with at least one bad index. Use the notation of Section~\ref{sec:private}. In particular, \eqref{eq:private-count} gives $\sum_i e_i\le8D-2n-h<n/2$. If there were three bad indices, \eqref{eq:bad-count} and the bound $e_i\ge D-a$ for the good index in \eqref{eq:private-count} would give $\sum_i e_i\ge2a+(D-a)=D+a\ge n/2$. Four bad indices would give $\sum_i e_i\ge3a>9n/16$.
Both conclusions contradict $\sum_i e_i<n/2$. There are therefore one or two bad indices.

\smallskip\noindent{\bf{Case 1:}} There is exactly one bad index in $I$.

Relabel $I$ so that index $1$ is bad. Let $E=e_2+e_3+e_4$ and $s_1=|N_H(v_1)\cap Q|$. By \eqref{eq:a4-external}, $q_1\le s_1<L$ since $P_1\cap Q\subseteq N_H(v_1)$. Every pair of good values has sum less than $D$. Indeed, a pair of good values with sum at least $D$, together with the third good value $e_i\ge D-a$ from \eqref{eq:private-count}, would give $E\ge2D-a$.
By \eqref{eq:bad-weight}, $e_1>D-5n/16+a-3q_1/4>-D/2+n/16+a$ since $q_1<L=2D-n/2$. Thus \eqref{eq:private-count} would give
$8D-2n\ge e_1+E>3D/2+n/16$, that is, $D>33n/104>5n/16$, contrary to \eqref{eq:a4-bounds}.

A vertex of $Q$ not adjacent to $v_1$ has neighbors only at good indices of $I$.  Exactly one neighbor would make that index bad; exactly two are excluded by Lemma~\ref{lem:two-exchange} and the preceding pairwise bound.  Such a vertex therefore has three neighbors in $I$.  Hence $h\ge a-s_1>a-L$. Also \eqref{eq:private-count}, with $e_i\ge D-a$, gives $h\le4D+4a-2n$.  Combining these bounds yields $a>5n/6-2D$. By \eqref{eq:private-count} and \eqref{eq:bad-weight}, $E\le8D-2n-h-e_1<7D-27n/16-a+3q_1/4-h\le7D-27n/16-2a+7s_1/4<21D/2-41n/16-2a$. It follows that $E-(2D+a-n/2)<17D/2-33n/16-3a
<29D/2-73n/16<-n/32$. Lemma~\ref{lem:three-exchange}, with $\Lambda=5n/4$, implies $Q\subseteq N_H(v_1)$, contrary to \eqref{eq:a4-external}.

\smallskip\noindent{\bf{Case 2:}} There are exactly two bad indices in $I$.

Relabel the bad indices as $1,2$ and let $E=e_3+e_4$. Equations~\eqref{eq:private-count}, \eqref{eq:bad-weight} and \eqref{eq:a4-bounds}, together
with $q_1+q_2\le a$, give
$E<6D-11n/8-2a+3(q_1+q_2)/4
\le6D-11n/8-5a/4<D$. A vertex of $Q$ avoiding $v_1,v_2$ would have one or two neighbors at good indices, contrary to the definition of a good index or to Lemma~\ref{lem:two-exchange}. Thus $Q\subseteq N_H(\{v_1,v_2\})$. Choose $v_2$, after relabelling if necessary, with at least $a/2$ neighbors in $Q$.  By \eqref{eq:a4-external}, there is a vertex $z\in Q$ nonadjacent to $v_2$. Let $W=V(H)\setminus N_H[\{v_2,z\}]$. Then $\alpha(H[W])\le2$. Moreover $Q\subseteq N_H[z]$, so $|N_H[v_2]\cap N_H[z]|\geq |N_H(v_2)\cap Q|\ge a/2$ and $|W|\ge n-2D+a/2$, contrary to Claim~\ref{cl:a4-forbidden}.
\end{proof}

\begin{theorem}\label{thm:alpha5}
If a graph $H$ satisfies $\alpha(H)\le5$, then $M_4(H)\ge|H|$.
\end{theorem}
\begin{proof}
Suppose a counterexample exists.  Among all counterexamples of the same order, choose $H$ with as few edges as possible. As in Theorem~\ref{thm:alpha4}, $\ov H$ is maximal $K_6$-free. Let $n=|H|$, $D=\Delta(H)+1$, and $a=\omega(H)$. Inequalities~\eqref{eq:elementary-M} and \eqref{eq:private-count} give
\begin{equation}\label{eq:a5-bounds}
 D<\frac n4,\qquad a<\frac n5,\qquad
 a\ge\frac{2n}{5}-D>\frac{3n}{20},\qquad \alpha(H)=5.
\end{equation}
The equality follows from $n\le\alpha(H)D$. These bounds also give $a>3D/5$.

\begin{claim}\label{cl:a5-forbidden}
If $\alpha(H[U])\le3$, then $|U|<n-2D+a/2$.
\end{claim}
\begin{proof}[Proof of claim~\ref{cl:a5-forbidden}]
Suppose $\alpha(H[U])\le3$ and $|U|\ge n-2D+a/2$. Extend $\ov{H[U]}$ to a maximal $K_4$-free graph $G$, and let $d=\delta(G)$. Then $d\ge|U|-D$ and $\alpha(G)\le\alpha(\ov{H[U]})\le a$. Lemma~\ref{lem:Lyle}(ii) gives $\alpha(G)\ge\min\{d/2,\,9d/4-|U|,\,(5d-2|U|)/3,\,5(2d-|U|)/2\}$. Since $D<n/4$ and $|U|\ge n-2D+a/2$, every expression in this minimum is greater than $a$ when $a\le n/6$. Thus $a>n/6$, and consequently $|U|>n/2+a/2>7n/12>7D/3$ and $d>4|U|/7$.
In the non-join alternative of Lemma~\ref{lem:Lyle}(i), $a\ge\alpha(G)\ge4d-2|U|\ge2|U|-4D\ge2n-8D+a>a$, a contradiction. In the join alternative of Lemma~\ref{lem:Lyle}(i), we have $G=G_1\vee G_2$ such that $G_1$ has no edges and $G_2$ is triangle-free. Then $|G_1|\le\alpha(G)\le a$ and $G_2$ is triangle-free. Since $\ov{G_2}$ is a subgraph of $H$, Lemma~\ref{lem:alpha2} and $M_4(H)<n$ imply $|G_2|<2n/5$. Hence $|U|<a+2n/5$, whereas $|U|>n/2+a/2$. These bounds give $a>n/5$, contrary to \eqref{eq:a5-bounds}.
\end{proof}

For nonadjacent vertices $r_1,r_2$ of $H$, the vertices outside
$N_H[r_1]\cup N_H[r_2]$ induce a graph with independence number at most three since $\alpha(H)=5$. By inclusion-exclusion and $|N_H[r_i]|\le D$ for $i=1,2$, at least $n-2D+|N_H[r_1]\cap N_H[r_2]|$ vertices lie outside $N_H[r_1]\cup N_H[r_2]$. Claim~\ref{cl:a5-forbidden} shows that this number is less than $n-2D+a/2$. Consequently, $|N_H[r_1]\cap N_H[r_2]|<a/2$. For a maximum clique $Q$ in $H$ and $x\notin Q$, maximality of $Q$ provides $y\in Q$ nonadjacent to $x$. Since $N_H(x)\cap Q\subseteq N_H[x]\cap N_H[y]$,
\begin{equation}\label{eq:a5-half}
 |N_H(x)\cap Q|<\frac a2.
\end{equation}
In particular, two vertices outside $Q$ cannot dominate $Q$.

If no edge of $H$ joins $Q$ to $V(H)\setminus Q$, then $\alpha(H-Q)\le4$ and Theorem~\ref{thm:alpha4} gives $M_4(H)\ge5(n-a)/4>n$, contrary to the counterexample assumption. Otherwise Lemma~\ref{lem:critical-edge} supplies a maximum independent set $I$ in $H$ disjoint from $Q$, with at least one bad index. Inequality~\eqref{eq:private-count} gives
$\sum_i e_i\le10D-2n-h<2D$. By \eqref{eq:private-count} and \eqref{eq:bad-count}, three bad indices give $\sum_i e_i\ge2a+2(D-a)=2D$; four give $\sum_i e_i\ge D+2a>2D$; and five give $\sum_i e_i\ge4a>2D$. Each case contradicts $\sum_i e_i<2D$. Hence there are one or two bad indices. For each bad index $i$, \eqref{eq:a5-half} gives $q_i<a/2$, so \eqref{eq:bad-weight} becomes
\begin{equation}\label{eq:a5-bad-deficit}
 e_i>D-\frac n4+a-\frac34q_i>D-\frac n4+\frac{5a}{8}.
\end{equation}

\smallskip\noindent{\bf{Case 1:}} There is exactly one bad index in $I$.

Relabel the bad index as $1$ and the good indices so that $e_2\le e_3\le e_4\le e_5$. Let $\sigma=e_2+e_3+e_4+e_5$ and $E=e_2+e_3+e_4$. Inequalities~\eqref{eq:private-count} and \eqref{eq:a5-bad-deficit} give
\begin{equation}\label{eq:a5-one-total}
 \sigma<9D-\frac{7n}{4}-\frac{5a}{8}-h.
\end{equation}
We show that $E<2D+a-2n/5$. Suppose first that $e_i+e_j<D$ for every pair of distinct good indices. A vertex $z\in Q$ nonadjacent to $v_1$ cannot have exactly one neighbor in $I$, since the corresponding index would be bad. If $z$ had exactly two neighbors $v_i,v_j$ in $I$, then both indices would be good and Lemma~\ref{lem:two-exchange} would imply $e_i+e_j\ge D$.
Thus every such $z$ has at least three neighbors in $I$. By \eqref{eq:a5-half}, more than $a/2$ vertices of $Q$ avoid $v_1$. Each such vertex has at least three neighbors in $I$ and contributes at least one to $h$, so $h>a/2$. It follows that $E\le3\sigma/4<27D/4-21n/16-27a/32$. Hence by \eqref{eq:a5-bounds}, we have
\[
 E-\left(2D+a-\frac{2n}{5}\right)
 <\frac{19D}{4}-\frac{73n}{80}-\frac{59a}{32}
 <\frac{11n}{40}-\frac{59a}{32}<-\frac n{640}.
\]
If some pair of good values has sum at least $D$, then $e_4+e_5\ge D$, $e_5\ge D/2$, and $\sigma\ge3D-2a$ by \eqref{eq:private-count}. Combining this lower bound with \eqref{eq:a5-one-total} and $D<n/4$ gives $11a/8>7n/4-6D>n/4$ by $h\geq 0$, and hence $a>2n/11$. If $E\ge2D+a-2n/5$, then $\sigma=E+e_5\ge5D/2+a-2n/5$. A second comparison with \eqref{eq:a5-one-total} gives $13a/8<13D/2-27n/20<11n/40$, so $a<11n/65<2n/11$, a contradiction. Thus $E<2D+a-2n/5$ in both cases. Lemma~\ref{lem:three-exchange}, applied to indices $2,3,4$, gives $Q\subseteq N_H(\{v_1,v_5\})$, contrary to \eqref{eq:a5-half}.

\smallskip\noindent{\bf{Case 2:}} There are exactly two bad indices in $I$.

Relabel the bad indices as $1,2$, let $E=e_3+e_4+e_5$, and define $s=|N_H(v_1)\cap Q|+|N_H(v_2)\cap Q|<a$. Then $q_1+q_2\leq s<a$. By \eqref{eq:private-count} and \eqref{eq:bad-weight}, we have
\begin{equation}\label{eq:a5-two-total}
 E<8D-\frac{3n}{2}-2a+\frac34(q_1+q_2)-h
 <8D-\frac{3n}{2}-\frac54a<2D-a.
\end{equation}
Since every good value is at least $D-a$ in \eqref{eq:private-count}, every pair of good values has sum at most $E-(D-a)<D$.
Clearly, at least $a-s$ vertices of $Q$ are nonadjacent to both $v_1$ and $v_2$. Every such vertex has exactly three neighbors in $I$, all at good indices: one neighbor is excluded by the definition of a good index, and two are excluded by Lemma~\ref{lem:two-exchange}. Thus $h\ge a-s$; also $q_1+q_2\le s$. Using $E\ge3(D-a)$ in \eqref{eq:a5-two-total} gives $3(D-a)\le E<8D-3n/2-3a+7s/4$, and hence $s>(6n-20D)/7>n/7$ since $D<n/4$.

Since $s<a$, choose $z\in Q$ adjacent to neither $v_1$ nor $v_2$, and let $W=V(H)\setminus N_H[\{v_1,v_2,z\}]$. The set $\{v_1,v_2,z\}$ is independent in $H$, so $\alpha(H[W])\le2$. Since $Q\subseteq N_H[z]$, we have $N_H[\{v_1,v_2,z\}]\subseteq N_H[z]\cup(N_H[v_1]\setminus Q)\cup(N_H[v_2]\setminus Q)$. Therefore
\[
 |N_H[\{v_1,v_2,z\}]|
 \le |N_H[z]|+|N_H[v_1]\setminus Q|+|N_H[v_2]\setminus Q|
 \le3D-s.
\]
Consequently, $|W|\ge n-3D+s>11n/28$. If $\overline{H[W]}$ is non-bipartite, Lemma~\ref{lem:13over5} gives $M_4(H)\ge13|W|/5>143n/140>n$ since $\alpha(H[W])\le2$, a contradiction. Thus $\ov{H[W]}$ is bipartite and $a\ge|W|/2>11n/56$. Returning to \eqref{eq:a5-two-total}, we obtain
\[
 E-\left(2D+a-\frac{2n}{5}\right)
 <6D-\frac{11n}{10}-\frac94a
 <\frac{2n}{5}-\frac94a<-\frac{47n}{1120}.
\]
Lemma~\ref{lem:three-exchange}, applied to indices $3,4,5$, now gives $Q\subseteq N_H(\{v_1,v_2\})$, again contrary to \eqref{eq:a5-half}.
\end{proof}

\begin{theorem}\label{thm:weighted-general}
For every integer $q\ge5$ and every graph $H$ with $\alpha(H)\le q$, we have $M_{q-1}(H)\ge|H|$.
\end{theorem}
\begin{proof}
The case $q=5$ is Theorem~\ref{thm:alpha5}. Suppose $q\ge6$ is the smallest parameter admitting a counterexample $H$, and let $n=|H|$, $D=\Delta(H)+1$, and $a=\omega(H)$. The counterexample assumption and \eqref{eq:elementary-M} give \begin{equation}\label{eq:general-counterexample}
 M_{q-1}(H)<n,\qquad n>(q-1)D,\qquad n>qa,
 \qquad\alpha(H)=q.
\end{equation}
For the equality, a maximum independent set in $H$ dominates $H$, so $n\le\alpha(H)D$; since $n>(q-1)D$, we have $\alpha(H)=q$. We first show that every maximum clique in $H$ meets every maximum independent set in $H$.

Suppose a maximum clique $Q$ in $H$ is disjoint from a maximum independent set $I=\{v_1,\ldots,v_q\}$ in $H$. Use the private-clique notation of Section~\ref{sec:private}, and put $\rho=a/D$. Since $a+D\geq 2n/q$ in \eqref{eq:private-count} and $n>(q-1)D$ in \eqref{eq:general-counterexample}, we have $a>(1-2/q)D$.  Furthermore, the first inequality of \eqref{eq:private-count} is less than $2D$ by $n>(q-1)D$, and less than $2q(D-a)$ by $n>qa$. Consequently,
\begin{equation}\label{eq:general-deficits}
 1-\frac2q<\rho<1,\qquad  \sum_i e_i<\min\{2D,\,2q(1-\rho)D\},\qquad  e_i\ge(1-\rho)D.
\end{equation}
The upper bound $\rho<1$ follows from $a<n/q\le D$. In particular, $\rho>2/3$. By \eqref{eq:bad-count}, the number $b$ of bad indices is at most three, since four bad indices would give $\sum_i e_i\ge3a>2D$. Hence there are at least three good indices. Choose the three good indices with the smallest values, denote their index set by $J$, and put $E=\sum_{i\in J}e_i$.

\begin{claim}\label{cl:three-smallest}
$E<D\min\{\rho+1/3,\,2-\rho\}$.
\end{claim}
\begin{proof}[Proof of claim~\ref{cl:three-smallest}]
The sum $E$ of the three smallest good values is at most $3/(q-b)$ times the sum of all $q-b$ good values. Indeed, each of the other $q-b-3$ good values is at least the largest of the selected three and hence at least $E/3$. Thus the sum of all good values is at least $(q-b)E/3$. For $b=0$, \eqref{eq:general-deficits} gives $E<6D/q\le D$. For $b=1$ and $q\ge7$, subtract the bad value, which is at least $(1-\rho)D$, from the bound $\sum_i e_i<2D$. This gives $E<3(1+\rho)D/(q-1)\le(1+\rho)D/2$. For $b=2$, \eqref{eq:bad-count} gives a sum of bad values at least $a=\rho D$, so $E<3(2-\rho)D/(q-2)\le3(2-\rho)D/4$. For $b=3$, the same argument gives $E<3(2-2\rho)D/(q-3)\le(2-2\rho)D$. Because $2/3<\rho<1$, each of these upper bounds is less than $D\min\{\rho+1/3,2-\rho\}$.

The remaining case is $q=6$ and $b=1$. Subtracting the bad value at least $(1-\rho)D$ from both bounds for $\sum_i e_i$ in
\eqref{eq:general-deficits} gives $E<\tfrac35\min\{1+\rho,11(1-\rho)\}D$. The first expression is less than $(\rho+1/3)D$ since $\rho>2/3$. For comparison with $(2-\rho)D$, if $\rho\le5/6$, then $2-\rho-3(1+\rho)/5>0$. If $\rho\ge5/6$, then
$2-\rho-33(1-\rho)/5>0$. Both comparisons are strict.
\end{proof}

Claim~\ref{cl:three-smallest} and Lemma~\ref{lem:three-exchange}(ii) imply that $I\setminus\{v_i:i\in J\}$ dominates $Q$.
Choose a minimal subset $T\subseteq I\setminus\{v_i:i\in J\}$ with $Q\subseteq N_H(T)$, and let $t=|T|$.
No single vertex outside the maximum clique $Q$ dominates $Q$, so $2\le t\le q-3$. For each $z\in T$, define $d_Q(z):=|N_H(z)\cap Q|$. Since $T$ dominates $Q$, $\sum_{z\in T}d_Q(z)\ge a$. Choose $v\in T$ minimizing $d_Q(v)$; then
$\sum_{z\in T\setminus\{v\}}d_Q(z)\ge(t-1)a/t$. Minimality of $T$ provides $w\in Q$ adjacent to $v$ and to no other
vertex of $T$. Thus $T'=(T\setminus\{v\})\cup\{w\}$ is independent in $H$.

Let $U=V(H)\setminus N_H[T']$, $u=|U|$, and $s=q-t\ge3$. Every independent set in $H[U]$ extends by $T'$, so $\alpha(H[U])\le s$. Since $Q\subseteq N_H[w]$, $N_H[T']\subseteq N_H[w]\cup\bigcup_{z\in T\setminus\{v\}}(N_H[z]\setminus Q)$.
Consequently,
\begin{equation}\label{eq:U-size}
 u\ge n-tD+\frac{t-1}{t}a,\qquad \alpha(H[U])\le s.
\end{equation}
Using $D<n/(q-1)$ and $a<n/q$ in \eqref{eq:general-counterexample}, we obtain
\begin{equation}\label{eq:u-sa}
 u-sa\ge n-tD-\left(s-1+\frac1t\right)a
 >\frac{n\bigl[(t-2)(q-3-t)+q-5\bigr]}{tq(q-1)}>0.
\end{equation}
Dividing \eqref{eq:U-size} by $D$ and using
$n/D>q-1$ from \eqref{eq:general-counterexample} and $a/D>1-2/q$ from \eqref{eq:general-deficits} also implies
\begin{equation}\label{eq:U-density}
 \frac uD>s-\frac1t-\frac{2(t-1)}{tq}.
\end{equation}

Suppose first that $s\ge4$, so $2\le t\le q-4$. Substitute $a\ge2n/q-D$ into \eqref{eq:U-size}, and then use $D<n/(q-1)$, to obtain
\begin{equation}\label{eq:U-proportion}
 u-\frac{s}{q}n\geq n-tD+\frac{t-1}{t}\left(\frac{2n}{q}-D\right)-\frac{s}{q}n
 >\frac{n\bigl[(t-2)(q-4-t)+q-6\bigr]}{tq(q-1)}\ge0.
\end{equation}
If $s=4$, Theorem~\ref{thm:alpha4} and \eqref{eq:scaling} give $M_{q-1}(H[U])\ge(q/5)M_4(H[U])\ge qu/4$. If $s\ge5$, the induction hypothesis, including the established case $s=5$, gives $M_{s-1}(H[U])\ge u$; then \eqref{eq:scaling} gives
$M_{q-1}(H[U])\ge qu/s$. In either case \eqref{eq:U-proportion} implies $M_{q-1}(H)\ge M_{q-1}(H[U])>n$, a contradiction.

It remains to consider $s=3$, so $t=q-3$. If $q\ge7$, \eqref{eq:U-density} gives $u/D>5/2$, since $
 3-\frac1{q-3}-\frac{2(q-4)}{q(q-3)}-\frac52
 =\frac{(q-7)(q-2)+2}{2q(q-3)}>0.$ Since $\Delta(H[U])+1\le D$, we have $u/(\Delta(H[U])+1)\geq u/D>5/2$, and thus  Lemma~\ref{lem:alpha3}(ii) applies to $H[U]$ and gives $\omega(H[U])\ge u/3>a$ by \eqref{eq:u-sa}. This contradicts $\omega(H)=a$.

For $q=6$ and $t=3$, we obtain $u>2D+2a/3$ and $u>3a$ by  \eqref{eq:general-counterexample}, \eqref{eq:U-size} and \eqref{eq:u-sa}. Since $\rho>2/3$, we have $u>22D/9>7D/3$. Since $\alpha(H[U])\leq s=3$, extend $\ov{H[U]}$ to a maximal $K_4$-free graph $G$. Then $\delta(G)\ge u-D>4u/7$ and $\alpha(G)\le\alpha(\ov{H[U]})\le a$. The non-join alternative of Lemma~\ref{lem:Lyle}(i) gives $\alpha(G)\ge 4\delta(G)-2u\geq 2u-4D>4a/3>a$, a contradiction.
In the join alternative of Lemma~\ref{lem:Lyle}(i), $G=G_1\vee G_2$ such that $G_1$ has no edges and $G_2$ is triangle-free.  Thus $|G_1|\le a$. Let $r=|G_2|$. Since $u>2D+2a/3$ and $a<D$, we have $r\ge u-a>2D-a/3>5D/3$. Also $r\ge u-a>2a$ because $u>3a$ in \eqref{eq:u-sa}.
Since $\ov{G_2}$ is a subgraph of $H$, $r-\delta(G_2)\le D$. Therefore $\delta(G_2)\ge r-D>2r/5$. Lemma~\ref{lem:triangle-degree}(ii) implies that $G_2$ is bipartite, so $\alpha(G)\ge r/2>a$, again a contradiction.

We have proved that every maximum clique in $H$ meets every maximum independent set in $H$. Choose a maximum clique $Q$ in $H$. Every maximum independent set in $H$ meets $Q$ in exactly one vertex; thus $\alpha(H-Q)=q-1$. Also $|H-Q|=n-a>(q-1)n/q$ since $a<n/q$ in \eqref{eq:general-counterexample}.
The induction hypothesis, including the case $q-1=5$, gives $M_{q-2}(H-Q)\ge|H-Q|$.
Finally, \eqref{eq:scaling} gives $
 M_{q-1}(H)\ge M_{q-1}(H-Q)
 \ge\frac q{q-1}M_{q-2}(H-Q)
 \ge\frac q{q-1}|H-Q|>n,$ contrary to \eqref{eq:general-counterexample}.
\end{proof}

\begin{proof}[Proof of Theorem~\ref{thm:main}]
If $k\ge6$ and $G$ is a nonempty $K_k$-free graph, then $\alpha(\ov G)\le k-1$. Theorem~\ref{thm:weighted-general}, with $q=k-1$, gives
$M_{k-2}(\ov G)\ge|G|$. Lemma~\ref{lem:binding-M} therefore implies $\bind(G)\le k-2$.

For sharpness, let $G$ be a complete $(k-1)$-partite graph whose parts all have cardinality $m\ge1$. This graph is $K_k$-free. A nonempty set meeting two distinct parts has neighborhood $V(G)$ and is not admissible in \eqref{eq:binding-definition}. Every admissible set lies in one part and has exactly $(k-2)m$ neighbors, so the neighborhood ratio is at least $k-2$, with equality for a whole part.  Hence the upper bound $k-2$ is sharp.
\end{proof}


\section*{AI Declaration}

ChatGPT was utilized to assist the authors in exploring proof strategies and refining the presentation of this manuscript. All mathematical arguments and proofs in the final manuscript were constructed and verified by the authors, who take full responsibility for all results.


\begin{thebibliography}{99}
	
	\bibitem{Anderson}
	I.~Anderson, \emph{Perfect matchings of a graph}, J. Combin. Theory Ser. B \textbf{10}(3) (1971), 183--186. \href{https://doi.org/10.1016/0095-8956(71)90041-4}{doi:10.1016/0095-8956(71)90041-4}.
	
	\bibitem{AES}
	B.~Andr\'asfai, P.~Erd\H{o}s, and V.~T.~S\'os, \emph{On the connection between chromatic number, maximal clique and minimal degree of a graph}, Discrete Math. \textbf{8}(3) (1974), 205--218. \href{https://doi.org/10.1016/0012-365X(74)90133-2}{doi:10.1016/0012-365X(74)90133-2}.
	
	\bibitem{Brandt}
	S.~Brandt, \emph{Dense graphs with bounded clique number}, Habilitation thesis, Freie Universit\"at, Berlin, 2001.
	
	\bibitem{DGLS}
	S.~Durocher, D.~S.~Gunderson, P.~C.~Li, and M.~Skala, \emph{Cycle-maximal triangle-free graphs}, Discrete Math. \textbf{338} (2015), 274--290. \href{https://doi.org/10.1016/j.disc.2014.10.002}{doi:10.1016/j.disc.2014.10.002}.
	
	\bibitem{GoddardSurvey}
	W.~Goddard, \emph{Binding Number, Cycles, and Cliques}, in R.~Gera, T.~W.~Haynes, and S.~T.~Hedetniemi (eds.), Graph Theory: Favorite Conjectures and Open Problems---2, Problem Books in Mathematics, Springer, Cham, 2018, pp.~19--25. \href{https://doi.org/10.1007/978-3-319-97686-0_3}{doi:10.1007/978-3-319-97686-0\_3}.
	
	\bibitem{Jin}
	G.~Jin, \emph{Triangle-free graphs with high minimal degrees}, Combin. Probab. Comput. \textbf{2} (1993), 479--490. \href{https://doi.org/10.1017/S0963548300000845}{doi:10.1017/S0963548300000845}.
	
	\bibitem{KM}
	V.~G.~Kane and S.~P.~Mohanty, \emph{Binding number, cycles and complete graphs}, in S.~B.~Rao (ed.), Combinatorics and Graph Theory, Lecture Notes in Mathematics \textbf{885}, Springer, Berlin--Heidelberg, 1981, pp.~290--296. \href{https://doi.org/10.1007/BFb0092273}{doi:10.1007/BFb0092273}.
	
	\bibitem{Lyle}
	J.~Lyle, \emph{A note on independent sets in graphs with large minimum degree and small cliques}, Electron. J. Combin. \textbf{21}(2) (2014), Paper P2.38. \href{https://doi.org/10.37236/3881}{doi:10.37236/3881}.
	
	\bibitem{LG}
	J.~Lyle and W.~Goddard, \emph{The binding number of a graph and its cliques}, Discrete Appl. Math. \textbf{157}(15) (2009), 3336--3340. \href{https://doi.org/10.1016/j.dam.2009.06.014}{doi:10.1016/j.dam.2009.06.014}.
	
	\bibitem{RW}
	A.~M.~Robertshaw and D.~R.~Woodall, \emph{Triangles and neighbourhoods of independent sets in graphs}, J. Combin. Theory Ser. B \textbf{80} (2000), 122--129. \href{https://doi.org/10.1006/jctb.2000.1974}{doi:10.1006/jctb.2000.1974}.
	
	\bibitem{ShiTriangle}
	R.~Shi, \emph{The binding number of a graph and its triangle}, Acta Math. Appl. Sin. Engl. Ser. \textbf{2}(1) (1985), 79--86. \href{https://doi.org/10.1007/BF01666521}{doi:10.1007/BF01666521}.
	
	\bibitem{ShiPancyclic}
	R.~Shi, \emph{The binding number of a graph and its pancyclism}, Acta Math. Appl. Sin. Engl. Ser. \textbf{3}(3) (1987), 257--269. \href{https://doi.org/10.1007/BF02007670}{doi:10.1007/BF02007670}.
	
	\bibitem{Woodall}
	D.~R.~Woodall, \emph{The binding number of a graph and its Anderson number}, J. Combin. Theory Ser. B \textbf{15}(3) (1973), 225--255. \href{https://doi.org/10.1016/0095-8956(73)90038-5}{doi:10.1016/0095-8956(73)90038-5}.
	
\end{thebibliography}
\end{document}